\documentclass[a4paper,10pt]{scrartcl}
\pdfoutput=1
\usepackage{ifpdf}
\usepackage[english]{babel}
\usepackage[utf8]{inputenc}
\usepackage[T1]{fontenc}
\usepackage{lmodern}
\usepackage{enumerate}
\usepackage{geometry}
\usepackage{amsfonts,amsmath,amssymb,amsopn}
\usepackage{amsthm}
\usepackage{mathtools}
\usepackage{stmaryrd}
\usepackage{nicefrac}
\usepackage{exscale}
\usepackage{empheq}
\usepackage{mathrsfs}

\usepackage[numbers,square]{natbib}
\usepackage{url} 
\usepackage[colorlinks=true,pdfpagelabels,unicode]{hyperref}
\hypersetup{linkcolor=blue, urlcolor=blue, citecolor=blue}
\usepackage[dvipsnames,svgnames,table]{xcolor}
\usepackage{orcidlink} 

\allowdisplaybreaks

\theoremstyle{plain}

\newtheoremstyle{theo}
	{3pt} 
	{3pt} 
	{\itshape} 
	{} 
		{\bfseries} 
	{\\} 
	{ } 
	{\thmname{#1}\thmnumber{ #2.}\thmnote{ - #3}} 
\theoremstyle{theo}

\newtheorem{definition}{Definition}[section]
\newtheorem{lemma}[definition]{Lemma}
\newtheorem{theorem}[definition]{Theorem}

\newenvironment{bew}{\begin{proof}[\bfseries Proof:]}{\end{proof}}

\newtheoremstyle{remark}
	{3pt} 
	{3pt} 
	{} 
	{} 
		{\bfseries} 
	{} 
	{ } 
	{\thmname{#1}\thmnumber{ #2.}\thmnote{ - #3}} 
\theoremstyle{remark}
\newtheorem{remark}[definition]{Remark}

\DeclareMathOperator{\bomega}{\overline{\Omega}}
\DeclareMathOperator{\romega}{\partial\Omega}

\DeclareMathOperator{\supp}{supp}
\DeclareMathOperator{\intd}{d\!}

\newcommand{\epsi}{\varepsilon}

\newcommand{\uep}{u_\epsi}

\newcommand{\vep}{v_\epsi}

\newcommand{\uu}{\underline{u}}

\newcommand{\GNI}{Gagliardo--Nirenberg inequality}

\newcommand{\into}[1]{\int_0^{#1}\!}

\newcommand{\intoT}{\into{T}}

\newcommand{\intomega}{\int_{\Omega}\!} 
\newcommand{\intoTomega}{\intoT\!\intomega}

\newcommand{\intinfomega}{\int_0^\infty\!\!\intomega}

\newcommand{\Lo}[1][1]{L^{#1}(\Omega)} 
\newcommand{\W}[1][1,2]{W^{#1}(\Omega)}

\newcommand{\LSp}[2]{L^{#1\;\!}\!\left(#2\right)} 

\newcommand{\LSpb}[2]{L^{#1\;\!}\!\big(#2\big)}
\newcommand{\LSploc}[2]{L_{loc}^{#1}\!\left(#2\right)} 

\newcommand{\LSplocb}[2]{L_{loc}^{#1}\big(#2\big)}
\newcommand{\WSp}[2]{W^{#1}\!\left(#2\right)}

\newcommand{\CSp}[2]{C^{#1}\!\left(#2\right)}

\newcommand{\CSpnl}[2]{C^{#1}\!\,(#2)} 

\newcommand{\R}{\mathbb{R}}
\newcommand{\N}{\mathbb{N}}

\newcommand{\dimN}{N}

\makeatletter
\def\@fnsymbol#1{\ensuremath{\ifcase#1\or *\or \ddagger\or \#\or
   \mathsection\or \mathparagraph\or \|\or **\or \dagger\dagger
   \or \ddagger\ddagger \else\@ctrerr\fi}}

\def\@fnsymbol#1{\ensuremath{\ifcase#1\or *\or \ddagger\or \#\or
   \mathsection\or \mathparagraph\or \|\or **\or \dagger\dagger
   \or \ddagger\ddagger \else\@ctrerr\fi}}
   
 \def\@lbibitem[#1]#2#3{%
  \if\relax\@extra@b@citeb\relax\else
    \@ifundefined{br@#2\@extra@b@citeb}{}{%
     \@namedef{br@#2}{\@nameuse{br@#2\@extra@b@citeb}}%
    }%
  \fi
  \@ifundefined{b@#2\@extra@b@citeb}{%
   \def\NAT@num{}%
  }{%
   \NAT@parse{#2}%
  }%
  \def\NAT@tmp{#1}%
  \expandafter\let\expandafter\bibitemOpen\csname NAT@b@open@#2\endcsname
  \expandafter\let\expandafter\bibitemShut\csname NAT@b@shut@#2\endcsname
  \@ifnum{\NAT@merge>\@ne}{%
   \NAT@bibitem@first@sw{%
    \@firstoftwo
   }{%
    \@ifundefined{NAT@b*@#2}{%
     \@firstoftwo
    }{%
     \expandafter\def\expandafter\NAT@num\expandafter{\the\c@NAT@ctr}%
     \@secondoftwo
    }%
   }%
  }{%
   \@firstoftwo
  }%
  {%
   \global\advance\c@NAT@ctr\@ne
   \@ifx{\NAT@tmp\@empty}{\@firstoftwo}{%
    \@secondoftwo
   }%
   {%
    \expandafter\def\expandafter\NAT@num\expandafter{\the\c@NAT@ctr}%
    \global\NAT@stdbsttrue
   }{}%
   \bibitem@fin
   \item[\href{#3}{\hfil\NAT@anchor{#2}{\NAT@num}}]
   \global\let\NAT@bibitem@first@sw\@secondoftwo
   \NAT@bibitem@init
  }%
  {%
   \NAT@anchor{#2}{}%
   \NAT@bibitem@cont
   \bibitem@fin
  }%
  \@ifx{\NAT@tmp\@empty}{%
    \NAT@wrout{\the\c@NAT@ctr}{}{}{}{#2}%
  }{%
    \expandafter\NAT@ifcmd\NAT@tmp(@)(@)\@nil{#2}%
  }%
}
\makeatother

\author{
Tobias Black\footnote{tblack@math.upb.de}\ \,\orcidlink{0000-0001-9963-0800}\\
{\small Institute of Mathematics,}\\[-5pt]
{\small Paderborn University,}\\[-5pt]
{\small 33098 Paderborn, Germany}
}
\title{Avoidance of caldera-type dead cores in a chemotaxis system with degenerate diffusion and compactly supported initial population density}
\date{}

\begin{document}
\maketitle
\begin{abstract}
\noindent
{\textbf{Abstract.} 
We consider a degenerate chemotaxis system of the form
\begin{align}\label{star}\tag{$\star$}
\left\lbrace
\begin{array}{r@{}l@{\quad}l}
&u_t=\nabla\cdot\big(D(u)\nabla u-uS(u)\nabla v\big)+f(u,v),\\
&v_t=\Delta v+g(u,v),\\
\end{array}\right.
\end{align}
in a bounded domain $\Omega\subset\mathbb{R}^{N}$ with smooth boundary subjected to no-flux and homogeneous Neumann boundary conditions. Herein, the diffusion coefficient $D\in C^0([0,\infty))\cap C^1((0,\infty))$ is assumed to satisfy $D(0)=0$ and $D'(s)\geq 0$ on $(0,\infty)$, and there are $s_0\in(0,1]$ and $d>0$ such that $D(s)\geq ds^{m-1}$ on $[0,s_0]$ and that
\begin{align*}
s D'(s)\leq C_D D(s)\quad\text{for }s\in[0,s_0].
\end{align*}
The sensitivity function $S\in C^2([0,\infty))$ and the source term $f\in C^{1}([0,\infty)\times[0,\infty))$ in the first equation are supposed to be nonnegative. The source term $g\in C^{1}([0,\infty)\times[0,\infty))$ of the second equation can in fact be negative. Prototypical choices for $g$ are $g(u,v)=-uv$ and $g(u,v)=-v+u$. \smallskip 

\noindent We show under suitable assumptions on weak solutions to \eqref{star} on $\Omega\times(0,T_0)$, that
whenever the smoothly bounded domain $\omega\subset\R^N$ and $T\in(0,T_0)$ are such that
\begin{align*}
\overline{\omega}\subseteq \Omega,\qquad u_0>0\ \text{ in }\ \overline{\omega},\qquad\text{ and }\qquad u>0\ \text{ on }\ \partial\omega\times(0,T),
\end{align*}
 then 
\begin{align*}
u>0\quad\text{in }\ \overline{\omega}\times[0,T).
\end{align*}
In particular, any dead cores that appear during the evolution must have developed from regions that were already part of the initial zero set.

}\medskip



\end{abstract}

\newpage
\section{Introduction}\label{sec1:intro}
Recapturing the visible patterns cell populations form in physical experiments has been a governing interest in the understanding of chemotaxis systems. The in-depth analysis undertaken in most of the prototypical mathematical models, however, did not uncover a process allowing plume-like aggregates as witnessed in \emph{Bacillus subtitlis} colonies when suspended in a drop of water (\cite{dombrowskiSelfConcentrationLargeScaleCoherence2004}), or the snowflake structures observed of \emph{B. subtilis} on agar plates (\cite{fujikawaFractalGrowthBacillus1989}) or \emph{Escherichia coli} (\cite{budreneComplexPatternsFormed1991}). Instead, the mathematical results established long-term spatial homogenization in the corresponding nutrient-taxis models, which quite certainly is not compatible with the experimental findings. (See \cite{TAO2011521,Winkler01022012,TAO20122520,JIANG.SIMA,LANKEIT2025113853} for corresponding results in a varied selection of related chemotaxis-consumption models and refer to the survey \cite{lankeitDepletingSignalAnalysis2023} for an additional overview.)\smallskip

Since an essential part of any visible pattern in a colony of bacteria consists of the areas devoid of any cells, we investigated in \cite{blackAbsenceDeadcoreFormations2025} whether solutions to the chemotaxis system
\begin{align*}
\left\lbrace
\begin{array}{r@{}l@{\quad}l}
&u_t=\nabla\cdot\big(D(u)\nabla u -uS(u)\nabla v\big)+f(u,v)\\
&v_t=\Delta v -uv,
\end{array}\right.
\end{align*}
arising from strictly positive initial data can produce zero-density zones inside the domain. In reaction-diffusion systems dedicated to chemical engineering (\cite{temkin1975diffusion,skrzypaczDeadcoreNondeadcoreSolutions2020}) the term \emph{dead core} was coined for open regions where the reactant concentration vanishes, since no chemical reaction can take place anymore inside this area. In \cite[Theorem 1.2]{blackAbsenceDeadcoreFormations2025} we showed that the classical solution to the system above remains strictly positive up to the maximal existence time. In particular, global solutions cannot form a dead core at all, which severely limits the possibility of interesting patterns to emerge. But since an overwhelming portion of realistic initial data goes beyond the assumed property of strict positivity, we now want to focus on the case where the bacterial density equals zero in a non-trivial part of the smoothly bounded domain $\Omega\subset\R^\dimN$. To be more precise, we will assume that, with some $\beta\in (0,1)$, 
\begin{align}\label{eq:IR}
\begin{cases}
u_0\in\W[1,\infty]\ &\text{is nonnegative with }u_0\not\equiv 0\ \text{and compactly supported in }\Omega, \\
v_0\in\CSp{2+\beta}{\bomega}\ &\text{is nonnegative with }v_0\not\equiv 0\ \text{and}\  \partial_\nu v_0=0\text{ on }\romega.
\end{cases}
\end{align}
For initial data with these properties, we will consider global weak solutions to an initial-boundary value problem of the form
\begin{align}\label{eq:CT-equation}
\left\lbrace
\begin{array}{r@{}l@{\quad}l@{\quad}l@{\,}c}
&u_t=\nabla\cdot\big(D(u)\nabla u -uS(u)\nabla v\big)+f(u,v)
,\ &x\in\Omega,& t\in(0,T),\\
&v_t=\Delta v +g(u,v),\ &x\in\Omega,& t\in(0,T),\\
&\big(D(u)\nabla u-u S(u)\nabla v\big)\cdot\nu=\nabla v\cdot\nu=0, &x\in\romega,& t\in(0,T),\\
&u(\cdot,0)=u_0,\quad v(\cdot,0)=v_0, &x\in\Omega,&
\end{array}\right.
\end{align}
where we assume that the diffusion coefficient $D\in\CSp{0}{[0,\infty)}\cap\CSp{1}{(0,\infty)}$ satisfies
\begin{align}\label{eq:cond-D}
D(0)=0,\qquad D(s)\geq ds^{m-1} \text{ on }\ [0,s_0],\qquad D'\geq 0\ \text{ on }\ (0,\infty)
\end{align}
for some $d>0$ and some $s_0\in(0,1]$, and that
\begin{align}\label{eq:cond-D-hoelder}
sD'(s)\leq C_D D(s)\ \text{ for }\ s\in[0,s_0].
\end{align}
The source term of the second equation is assumed to be of class $g\in\CSp{1}{[0,\infty)\times[0,\infty)}$ and the remaining constitutive functions of the first equation are supposed to be nonnegative and of class
\begin{align*}
&S\in\CSp{2}{[0,\infty)}\ \text{ and }\ f\in\CSp{1}{[0,\infty)\times[0,\infty)}.
\end{align*}
For $M>0$ we then denote by $C_S(M)>0$, $C_f(M)>0$ and $C_g(M)>0$ constants such that
\begin{align}\label{eq:cond-S_f}
S&\leq C_S(M)\ \text{ on }\ [0,M]\nonumber\\ \text{ as well as }\quad f&\leq C_f(M)\quad\text{and}\quad |g|\leq C_g(M) \text{ on }[0,M]\times[0,M]
\end{align}
and introduce the positive constants
\begin{align}\label{eq:ID-bounds}
K_{u_0}:=\|u_0\|_{\W[1,\infty]}\quad\text{and}\quad K_{v_0}:=\|v_0\|_{\CSp{2+\beta}{\bomega}},
\end{align}
which will later appear in our proofs. The most prototypical choices for all these functions consist of $D(s)=s^{m-1}$ for some $m>1$, $S(u)=1$, $f(s)=0$, and either $g(u,v)=-uv$ or $g(u,v)=-v+u$, for which \eqref{eq:CT-equation} takes the form of Keller--Segel models with porous medium type diffusion and signal consumption or production model, respectively.\smallskip

To further motivate our investigation, think about a setting where initially the cell population is centered in the middle of the domain with an outer ring of empty space around it. In this case, the dead-core set would only consist of a single connected component and one question we want to address is whether a constellation can occur where a new connected component disjoint from the original dead core appears within the support of the bacterial population. This is what we will call a \emph{caldera-type dead core}, since in a two-dimensional setting the plotted cell density would resemble the image of a small volcano.\medskip

\textbf{Main results.} The focal point of our analysis is the study of global weak solutions to \eqref{eq:CT-equation} in the following sense.
\begin{definition}\label{def:weak-sol}
Suppose $(u_0,v_0)$ satisfies \eqref{eq:IR}, let $T\in(0,\infty]$ and set $\Phi(s):=\int_0^s D(\sigma)\intd\sigma$. Then the pair of nonnegative functions $(u,v)$ will be called a weak solution of \eqref{eq:CT-equation} in $\Omega\times(0,T)$ if
\begin{align*}
u&\in \LSploc{1}{\bomega\times[0,T)}
,\qquad
v\in \LSploc{\infty}{\bomega\times[0,T)}\cap\LSploc{1}{[0,T);\W[1,1]}
\end{align*}
with $\Phi(u)\in\LSploc{1}{[0,T);\WSp{1,1}{\Omega}}$ and $u\nabla v\in\LSploc{1}{\bomega\times[0,T)}$ and if
\begin{align*}
-\intoTomega u\varphi_t-\intomega u_0\varphi(\cdot,0)=-\intoTomega\nabla\Phi(u)\cdot\nabla\varphi+\intoTomega uS(u)\nabla v\cdot\nabla\varphi+\intoTomega f(u,v)\varphi,
\end{align*}
and
\begin{align*}
-\intoTomega v\varphi_t-\intomega v_0\varphi(\cdot,0)=-\intoTomega\nabla v\cdot\nabla\varphi+\intoTomega g(u,v)\varphi
\end{align*}
hold for every $\varphi\in C^{\infty}_0\big(\bomega\times[0,T)\big)$. If $T=\infty$ the solution will be called a global weak solution of \eqref{eq:CT-equation}.
\end{definition}

We will show that under some mild additional assumptions on regularity and boundedness, any such global weak solution satisfies the following:
\begin{theorem}\label{theo1}
Let $\Omega\subset\R^\dimN$ be a smoothly bounded domain and $m>1$. Assume that $(u_0,v_0)$ satisfy \eqref{eq:IR} with some $\beta\in(0,1)$. Suppose that $D\in\CSp{0}{[0,\infty)}\cap \CSp{2}{(0,\infty)}$ satisfies \eqref{eq:cond-D} and \eqref{eq:cond-D-hoelder}, that $S\in\CSp{2}{[0,\infty)}$ and that $f\in \CSp{1}{[0,\infty)\times[0,\infty)}$ is nonnegative. Let $T_0\in(0,\infty]$ and assume that $(u,v)$ is a weak solution of \eqref{eq:CT-equation} in $\Omega\times(0,T_0)$ in the sense of Definition~\ref{def:weak-sol}, satisfying
\begin{align}\label{eq:cond-conti}
u&\in C^0\big([0,T_0);\LSp{2}{\Omega}\!\big)\cap\LSpb{\infty}{\Omega\times(0,T_0)}
\quad\ \text{and}\quad\ 
v\in C^0\big([0,T_0);\LSp{2}{\Omega}\!\big)\cap\LSpb{\infty}{\Omega\times(0,T_0)}.
\end{align}
 Then, for each $T\in(0,T_0)$ there exists $\theta=\theta(T)\in(0,1)$ such that
$$(u,v)\in \CSpnl{\theta,\frac{\theta}{2}}{\bomega\times[0,T]}\times \CSpnl{2+\theta,1+\frac{\theta}{2}}{\bomega\times[0,T]},$$
and whenever the smoothly bounded domain $\omega\subset\R^\dimN$ and $T\in(0,T_0)$ are such that
\begin{align*}
\overline{\omega}\subseteq \Omega,\qquad u_0>0\ \text{ in }\ \overline{\omega},\qquad\text{ and }\qquad u>0\ \text{ on }\ \partial\omega\times(0,T),
\end{align*}
 then 
\begin{align*}
u>0\quad\text{in }\ \overline{\omega}\times[0,T).
\end{align*}
\end{theorem}

To put this into more practical terms, Theorem~\ref{theo1} illustrates that dead cores can only expand into the support of $u$ from the parts of the domain, where cell density is already zero. (See also \cite{Fischer2013,xuChemotaxisModelDegenerate2020,Fuest_Heihoff_2025} for related prototypical models where initial support shrinking does occur.)  In particular, all dead cores that are present at some time of the evolution have grown out of the closed set $\{u_0=0\}$ in a continuous manner. As a consequence, in the one-dimensional setting or in the radial setting of arbitrary dimension the Caldera-type dead cores can only form if $\mathrm{int}(\supp(u_0))\cap\{u_0=0\}\neq\emptyset$. That is, as long as the initial distribution does not already have crater-like features no new crater structures can form. In particular, if the interior of the support of $u_0$ is a ball around the origin the support of the bacterial population is limited to behave like a shrinking or expanding ball and the dead-core set can be described as an outer annular region.\medskip

To further illustrate the findings of the theorem above, let us take a closer look at one simple but relevant prototypical system. In fact, in the second part of the work we are going to consider the system
\begin{align}\label{eq:CT-prototype-choice}
\left\lbrace
\begin{array}{r@{}l@{\quad}l@{\quad}l@{\,}c}
&u_t=\nabla\cdot\big(\nabla u^{m} - u\nabla v\big)
,\ &x\in\Omega,& t>0,\\
&v_t=\Delta v -uv,\ &x\in\Omega,& t>0,\\
&\big(u^{m-1}\nabla u-u\nabla v\big)\cdot\nu=\nabla v\cdot\nu=0, &x\in\romega,& t>0,\\
&u(\cdot,0)=u_0,\quad v(\cdot,0)=v_0, &x\in\Omega,&
\end{array}\right.
\end{align}
with $m>1$ in a smoothly bounded convex domain $\Omega\subset\R^\dimN$, $\dimN\in\{2,3\}$. We will find that indeed this system has a weak solution in the sense of Definition~\ref{def:weak-sol}, which is global in time and also bounded as required by Theorem~\ref{theo1}. From an application of Theorem~\ref{theo1} we therefore obtain the following result.

\begin{theorem}\label{theo2}
Let $m>1$, assume that $(u_0,v_0)$ satisfy \eqref{eq:IR} with some $\beta\in(0,1)$. The problem \eqref{eq:CT-prototype-choice} has a global weak solution in the sense of Definition~\ref{def:weak-sol}, which is bounded in the sense that for some $M>0$
\begin{align*}
\big\|u(\cdot,t)\big\|_{\Lo[\infty]}+\big\|v(\cdot,t)\big\|_{\Lo[\infty]}\leq M\quad\text{for all }t>0.
\end{align*}
Moreover, for each $T>0$ there exists $\theta=\theta(T)\in(0,1)$ such that
$$(u,v)\in \CSpnl{\theta,\frac{\theta}{2}}{\bomega\times[0,T]}\times \CSp{2+\theta,1+\frac{\theta}{2}}{\bomega\times[0,T]},$$
and whenever the smoothly bounded domain $\omega\subset\R^\dimN$ and $T>0$ are such that
\begin{align*}
\overline{\omega}\subseteq \Omega,\qquad u_0>0\ \text{ in }\ \overline{\omega},\qquad\text{ and }\qquad u>0\ \text{ on }\ \partial\omega\times(0,T),
\end{align*}
 then 
\begin{align*}
u>0\quad\text{in }\ \overline{\omega}\times[0,T).
\end{align*}
\end{theorem}

Accordingly, the system \eqref{eq:CT-prototype-choice}, which features a chemotaxis process involving a prototypical porous medium type diffusion and consumption of signal chemical, provides one explicit first example, where a dead core has to expand from already established zero-density zones.\medskip

\textbf{Details on the approach.} The core of the argument will consist of a comparison theorem with a suitably chosen subsolution. Since $u_0$ may in fact be zero in non-trivial parts of the domain, however, we can not expect to employ the comparison theorem on the whole of $\Omega$ in any meaningful way if we want to conclude something about strict positivity of solutions. Moreover, the degeneracy of the diffusion coefficient also does not work in favor of obtaining solutions sufficiently regular to warrant the application of a comparison argument in the first place. At most we could expect some Hölder regularity for the first solution component (Lemma~\ref{lem:u-hoelder}). The condition $D(s)\geq ds^{m-1}$ in \eqref{eq:cond-D} is introduced for precisely this step in the argument to ensure that $\Phi$, as given in Definition~\ref{def:weak-sol}, has a Hölder continuous inverse. There are other conditions one could impose on $D$ to ensure this outcome. In fact, one could replace it as follows (see also \cite{TB23_hoeldertaxis}).

\begin{remark}
The condition $D(s)\geq ds^{m-1}$ in \eqref{eq:cond-D} could be replaced by requiring that there are $p>0$ and $C_D>0$ such that $$\int_0^{s_0}\!\frac{1}{|D(\sigma)|^p}\intd\sigma\leq C_D.$$
\end{remark}

The continuity of $u$ then allows for additional regularity arguments inside $\{u>0\}$ (Lemma~\ref{lem:u-classical-in-positivity-set}), which in turn fuel the local comparison undertaken in the proof of Theorem~\ref{theo1}. In Section~\ref{sec3} we will derive additional a priori estimates which ensure that the weak solution for the prototypical system \eqref{eq:CT-prototype-choice} actually satisfies the conditions in \eqref{eq:cond-conti}, which go beyond weak solution properties commonly established in the literature.

\setcounter{equation}{0}
\section{Dead cores can only grow from already established zero-density zones of the bacterial population}\label{sec2}
This section will be dedicated to our proof of Theorem~\ref{theo1}. Throughout this section, we fix a smoothly bounded $\Omega\subset\R^\dimN$, initial data $(u_0,v_0)$ satisfying \eqref{eq:IR} and $T_0\in(0,\infty]$ and denote by $(u,v)$ a weak solution of \eqref{eq:CT-equation} in $\Omega\times(0,T_0)$ in the sense of Definition~\ref{def:weak-sol}, which additionally satisfies the continuity property in \eqref{eq:cond-conti} and is bounded in the sense that for some $M>0$
\begin{align}\label{eq:assumption-bdd}
\|u(\cdot,t)\|_{\Lo[\infty]}+\|v(\cdot,t)\|_{\Lo[\infty]}\leq M\quad\text{on }(0,T_0).
\end{align}
In a first step we establish an $\Lo[\infty]$-bound for $\nabla v$.

\begin{lemma}\label{lem:v-nab-infty}
Assume that $(u,v)$ is a bounded weak solution of \eqref{eq:CT-equation} in $\Omega\times(0,T_0)$ such that \eqref{eq:cond-conti} and \eqref{eq:assumption-bdd} hold.
There is $K>0$ such that $$\|\nabla v(\cdot,t)\|_{\Lo[\infty]}\leq K\quad\text{for all }t\in(0,T_0).$$
\end{lemma}

\begin{bew}
Denoting by $(e^{\sigma\Delta})_{\sigma\geq0}$ the Neumann heat semigroup and by $\lambda_1>0$ the first nonzero eigenvalue of $-\Delta$ in $\Omega$ under Neumann boundary conditions, we find from expressing $v$ by its variation-of-constants representation that
\begin{align*}
\|\nabla v(\cdot,t)\|_{\Lo[\infty]}\leq\|\nabla e^{t\Delta}v_0\|_{\Lo[\infty]}+\int_0^t\big\|\nabla e^{-(t-s)\Delta} g\big(u(\cdot,s),v(\cdot,s)\big)\big\|_{\Lo[\infty]}\intd s\quad\text{for all }t\in(0,T_0).
\end{align*}
Drawing on well-known smoothing properties (\cite[Lemma 1.3]{win10jde}) of this semigroup, we find $C_1>0$ satisfying
\begin{align}\label{eq:sem-group-est-v}
\|\nabla v(\cdot,t)\|_{\Lo[\infty]}\leq\|\nabla e^{\Delta t}v_0\|_{\Lo[\infty]}+C_1\int_0^t\big(1+(t-s)^{-\frac12}\big)e^{-\lambda_1 t}\big\|g\big(u(\cdot,s),v(\cdot,s)\big)\big\|_{\Lo[\infty]}\intd s
\end{align}
for all $t\in(0,T_0)$. With $C_g(M)>0$ as chosen in \eqref{eq:cond-S_f} we find from \eqref{eq:assumption-bdd} that
\begin{align*}
\big\|g\big(u(\cdot,s),v(\cdot,s)\big)\big\|_{\Lo[\infty]}\leq C_g(M)\quad\text{on }(0,T_0),
\end{align*}
so that plugging this back into \eqref{eq:sem-group-est-v} and drawing on \eqref{eq:ID-bounds} entails that
\begin{align*}
\big\|\nabla v(\cdot,t)\big\|_{\Lo[\infty]}\leq K_{v_0}+C_1 C_g(M)\frac{1+\sqrt{\lambda_1\pi}}{\lambda_1}>0\quad\text{on }(0,T),
\end{align*}
completing the proof.
\end{bew}

Control on the convection term at hand, we can now consider the Hölder regularity of $u$.

\begin{lemma}\label{lem:u-hoelder}
Assume that $(u,v)$ is a weak solution of \eqref{eq:CT-equation} in $\Omega\times(0,T_0)$ such that \eqref{eq:cond-conti} and \eqref{eq:assumption-bdd} hold. There are $\theta\in(0,1)$ and $C>0$ such that if $T\in(0,T_0)$, then 
\begin{align}\label{eq:u-hoelder}
\|u(\cdot,t)\|_{\CSpnl{\theta,\frac{\theta}{2}}{\bomega\times[0,T]}}\leq C.\end{align}
\end{lemma}

\begin{bew}
With $\Phi(s)=\int_0^s D(\sigma)\intd \sigma$, $a(x,t)=-S(u)\nabla v$ and $b(x,t)=f(u,v)$, the function $u$ is a weak solution of
\begin{align*}
u_{t}=\Delta\Phi(u)+\nabla\cdot\big(a(x,t)u)+b(x,t)\quad\text{in }\Omega\times(0,T).
\end{align*}
The properties of $D$ imply $\Phi\in\CSp{0}{[0,\infty)}\cap\CSp{2}{(0,\infty)}$ with $\Phi(0)=0$ and $\Phi'>0$ on $(0,\infty).$ Moreover, there are $s_0>0$ and $C_1>0$ such that
\begin{align*}
s\Phi''(s)\leq C_1\Phi'(s)\quad\text{on }[0,s_0]
\end{align*}
and that
\begin{align*}
\big\|\Phi\big(u(\cdot,t)\big)\big\|_{\Lo[\infty]}\leq \Phi(M)\quad\text{for all }t\in(0,T).
\end{align*}
Additionally, in view of \eqref{eq:cond-S_f} and Lemma~\ref{lem:v-nab-infty} there are $K>0$, $C_S=C_S(M)>0$ and $C_f=C_f(M)>0$ such that
\begin{align*}
\big\|S(u)\nabla v\big\|_{\LSp{\infty}{\Omega\times(0,T)}}\leq C_S K\quad\text{and}\quad\big\|f(u,v)\big\|_{\LSp{\infty}{\Omega\times(0,T)}}\leq C_f.
\end{align*}
Hence, introducing $C_2=C_2(M):=C_S K+C_f>0$ we have
\begin{align*}
\|a\|_{\LSp{\infty}{\Omega\times(0,T);\R^\dimN}}+\|b\|_{\LSp{\infty}{\Omega\times(0,T)}}\leq C_2.
\end{align*}
Accordingly, as $\Phi^{-1}$ is Hölder continuous on $[0,\Phi(M)]$ and $\|u_0\|_{\W[1,\infty]}\leq K_{u_0}$ entails $u_0\in\CSp{\theta_1}{\bomega}$ for some $\theta_1\in(0,1)$, the requirements of \cite[Theorem 1.8 and Corollary 1.9]{TB23_hoeldertaxis} are satisfied and the results provide $C>0$ and $\theta\in(0,1)$ such that (up to a re-definition of $u$ on a null set of $\Omega\times(0,T)$) \eqref{eq:u-hoelder} holds.
\end{bew}

With the lemma above, the coefficients in the second equation of \eqref{eq:CT-equation} are all Hölder continuous and parabolic Schauder estimates allow for a regularity improvement on $v$.

\begin{lemma}\label{lem:v-hoelder}
Assume that $(u,v)$ is a weak solution of \eqref{eq:CT-equation} in $\Omega\times(0,T_0)$ such that \eqref{eq:cond-conti} and \eqref{eq:assumption-bdd} hold. There are $\theta\in(0,1)$ and $C>0$ such that if $T\in(0,T_0)$, then
\begin{align}\label{eq:v-2-hoelder}
\big\|v(\cdot,t)\big\|_{\CSpnl{2+\theta,1+\frac{\theta}{2}}{\bomega\times[0,T]}}\leq C.
\end{align}
Moreover, one can find $L>0$ such that
$$\big\|\Delta v(\cdot,t)\big\|_{\Lo[\infty]}\leq L\quad\text{for all }t\in(0,T).$$
\end{lemma}

\begin{bew}
Letting $b(x,t)=g\big(u(x,t),v(x,t)\big)$ the function $v$ obviously solves
\begin{align}\label{eq:hölder-v-eq}
v_{t}=\Delta v+b(x,t)\quad\text{weakly in }\Omega\times(0,T),
\end{align}
where $b$ fulfills $\|b\|_{\LSp{\infty}{\Omega\times[0,T]}}\leq C_g(M)$ with $C_g(M)>0$ as in \eqref{eq:cond-S_f}. As $\|v_0\|_{\CSp{\beta}{\bomega}}\leq K_{v_0}$, we conclude from \cite[Theorem 1.3]{PorzVesp93} that there are $\theta_1\in(0,1)$ and $C_1>0$ such that
\begin{align*}
\|v\|_{\CSpnl{\theta_1,\frac{\theta_1}{2}}{\bomega\times[0,T]}}\leq C_1.
\end{align*}
Combining with Lemma~\ref{lem:u-hoelder}, we notice that the equation in \eqref{eq:hölder-v-eq} is governed by a smooth linear parabolic differential operator with forcing term of class $\CSp{\theta_2,\frac{\theta_2}{2}}{\bomega\times[0,T]}$ for some $\theta_2\in(0,1)$. Since, moreover, $\|v_0\|_{\CSp{2+\beta}{\bomega}}\leq K_{v_0}$, parabolic Schauder theory (cf. \cite[Theorem III.5.1 and Theorem IV.5.3]{LSU}) entails the existence of $\theta_3\in(0,1)$ and $C_2>0$ such that
\begin{align*}
\|v\|_{\CSpnl{2+\theta_3,1+\frac{\theta_3}{2}}{\bomega\times[0,T]}}\leq C_2.
\end{align*}
The desired bound on $\Delta v$ with $L=C_2>0$ is an evident consequence thereof.
\end{bew}

For our main argument, we intend to employ comparison arguments inside a space-time cylinder which is contained in $$\big\{u>0\big\}:=\big\{(x,t)\in\Omega\times(0,T)\,\vert\, u(x,t)>0\big\}.$$ In order to make the first equation of \eqref{eq:CT-equation} accessible for said comparison arguments, we have to ensure that also $u$ is sufficiently regular inside $\{u>0\}$.

\begin{lemma}\label{lem:u-classical-in-positivity-set}
Assume that $(u,v)$ is a weak solution of \eqref{eq:CT-equation} in $\Omega\times(0,T_0)$ such that \eqref{eq:cond-conti} and \eqref{eq:assumption-bdd} hold and let $T\in(0,T_0)$. Then, there is $\theta\in(0,1)$ such that
\begin{align}\label{eq:u-c21-pos}
u\in\CSp{2+\theta,1+\frac{\theta}{2}}{\{u>0\}}.
\end{align}
\end{lemma}

\begin{bew}
In view of Lemma~\ref{lem:u-hoelder} $u$ is continuous in $\bomega\times[0,T]$, so $\{u>0\}$ is open. Accordingly, for any arbitrary point $(x_0,t_0)\in\{u>0\}$ we can choose $r>0$ and $\eta>0$ such that $$\overline{B_{2r}}(x_0)\times[t_0-2\eta,t_0]\subset \{u>0\}$$ and such that $$\delta:=\min\big\{u(x,t)\,\vert\, (x,t)\in\overline{B_{2r}}(x_0)\times[t_0-2\eta,t_0]\big\}>0$$ is a well-defined positive number. Introducing $\tilde{D}\in \CSp{1}{[0,\infty)}$ such that $\tilde{D}(s)\geq(\frac{\delta}{2})^{m-1}$ for all $s\in[0,\infty)$ and such that $\tilde{D}(s)=s^{m-1}$ for $s\geq \delta$, we see that $u$ is a weak solution of $$u_t=\nabla\cdot\big(a(x,t)\nabla u-b(x,t)u\big)+c(x,t)\quad\text{in } B_{2r}(x_0)\times(t_0-2\eta,t_0)$$ with $a(x,t):=\tilde{D}(u)$, $b(x,t):=S(u)\nabla v$ and $c(x,t):=f(u,v)$, where the functions $a,b,c$ are Hölder continuous according to $\tilde{D}\in\CSp{1}{[0,\infty)}$, \eqref{eq:cond-S_f} and the results from Lemma~\ref{lem:u-hoelder} and Lemma~\ref{lem:v-hoelder}. Since this modified equation is uniformly parabolic in $B_{2r}(x_0)\times(t_0-2\eta,t_0)$, we can employ interior parabolic Schauder theory (e.g. \cite[Chapter IV-V]{liebermanSecondOrderParabolic1996}) to find that for some $\theta\in(0,1)$
$$u\in\CSp{2+\theta,1+\frac{\theta}{2}}{\overline{B_{r}}(x_0)\times[t_0-\eta,t_0]}.$$ Since $(x_0,t_0)\in\{u>0\}$ was arbitrary, we conclude \eqref{eq:u-c21-pos}.
\end{bew}

Sufficient regularity inside the superlevel set at hand, we can rely on comparison arguments for space-time cylinders inside $\{u>0\}$ to prove the first main theorem.

\begin{proof}[\textbf{Proof of Theorem~\ref{theo1}:}]
We fix a smoothly bounded domain $\omega\subset\R^\dimN$ and $T\in(0,T_0)$ such that 
\begin{align*}
\overline{\omega}\subseteq \Omega,\qquad u_0>0\ \text{ in }\ \overline{\omega},\qquad\text{ and }\qquad u>0\ \text{ on }\ \partial\omega\times(0,T)
\end{align*}
and claim that $u>0$ in $\overline{\omega}\times[0,T)$. Assuming the contrary, there would be some $(x_0,t_0)\in \omega\times(0,T)$ such that $u(x_0,t_0)=0$ and $u>0$ in $\omega\times(0,t_0)$, where we used that by our assumptions on $\omega$ and $T$ and the continuity of $u$ in $\bomega\times[0,T]$ we have $t_0>0$ and $x_0\not\in \partial\omega$. By the assumed positivity of $u$ in $\omega\times(0,t_0)$ and Lemma~\ref{lem:u-classical-in-positivity-set} we would have $u\in\CSp{2,1}{\omega\times(0,t_0)}$. Introducing the positive numbers $$\delta_{\overline{\omega}}^0:=\min_{\overline{\omega}} u_0>0,\qquad \delta_{\partial\omega}^{t_0}:=\min_{\partial\omega\times[0,t_0]} u>0\quad\text{and}\quad\delta_{\Sigma_{\omega}^{t_0}}:=\min\Big\{\delta_{\overline{\omega}}^0,\;\delta_{\partial\omega}^{t_0}\Big\}$$ and the spatially homogeneous function
$$\uu(t):=\delta_{\Sigma_{\omega}^{t_0}}e^{-C_s L t},\quad t\geq0,$$
with $C_S>0$ and $L>0$ provided by \eqref{eq:cond-S_f} and Lemma~\ref{lem:v-hoelder}, respectively, we would find that 
\begin{align*}
\uu_t-\nabla\cdot\big(D(\uu)\nabla\uu-\uu S(\uu)\nabla v\big)-f(\uu,v)=-C_S L\uu+\uu S(\uu)\Delta v-f(\uu,v)\leq 0\quad\text{in }\omega\times(0,t_0). 
\end{align*}
Moreover, $\uu(0)\leq\delta_{\overline{\omega}}^0\leq u_0$ and also $\uu_{\vert\partial\omega\times(0,t_0)}\leq \delta_{\partial\omega}^{t_0}\leq u_{\vert\partial\omega\times(0,t_0)}$. Hence, drawing on the parabolic comparison principle for Dirichlet data in $\omega\times(0,t_0)$, we could find that actually $$u\geq \delta_u:=\delta_{\Sigma_{\omega}^{t_0}} e^{-C_s L t_0}>0\quad\text{in }\overline{\omega}\times[0,t_0),$$ which by continuity of $u$ in $\bomega\times[0,T]$ would contradict $u(x_0,t_0)=0$. Accordingly, $u$ has to be positive throughout all of $\overline{\omega}\times[0,T)$.
\end{proof}

\setcounter{equation}{0} 
\section{Dead-core expansion in a prototypical chemotaxis system}\label{sec3}
In the final section we are investigating whether the conditions of Theorem~\ref{theo1} can indeed be met by chemotaxis systems commonly studied in the literature. For this purpose let us consider the prototypical system \eqref{eq:CT-prototype-choice} obtained from \eqref{eq:CT-equation} by the specific choices $D(u)=u^{m-1}$ with $m>1$, $S(u)\equiv 1$, $f(u,v)\equiv0$ and $g(u,v)=-uv$. In this framework our system is a close relative to the more general settings considered in the studies \cite{caoGlobalintimeBoundedWeak2014,Win-ct_fluid_3d-CPDE15,winklerChemotaxisStokesInteractionVery2022}. There, the results even cover a tensor-valued sensitivity function and the latter two also include a solution component for fluid interaction, both of which go beyond our scope here. Nevertheless, we may draw on these references to conclude the existence of a global bounded weak solution for our prototypical setting. In fact, the weak solution is obtained as limit of the classical solutions to the family of regularized systems
\begin{align}\label{eq:approx-prop}
\left\lbrace
\begin{array}{r@{}l@{\quad}l@{\quad}l@{\,}c}
&u_{\epsi t}=\nabla\cdot\big(\nabla(\uep+\epsi)^{m} - \frac{\uep}{1+\epsi\uep}\nabla \vep\big)
,\ &x\in\Omega,& t>0,\\
&v_{\epsi t}=\Delta \vep -\uep\vep,\ &x\in\Omega,& t>0,\\
&\nabla\uep\cdot\nu=\nabla \vep\cdot\nu=0, &x\in\romega,& t>0,\\
&\uep(\cdot,0)=u_{0},\quad \vep(\cdot,0)=v_{0}, &x\in\Omega.&
\end{array}\right.
\end{align}
While technically the approximate systems in the referenced works differ slightly from the one above, in particular additionally requiring a spatial cut-off for the chemotactic sensitivity near $\romega$, the obtained results easily transfer to our system family. Our main objective in this section will be to show that the weak solutions obtained in this way actually feature the continuity property stated in \eqref{eq:cond-conti}, which is a crucial and necessary trait for the application of Theorem~\ref{theo1}.

Prior to tackling this continuity property, let us briefly gather the results obtained in \cite{caoGlobalintimeBoundedWeak2014,Win-ct_fluid_3d-CPDE15,winklerChemotaxisStokesInteractionVery2022} and sketch the main steps necessary for the corresponding proof. While the approach slightly differs depending on space dimension and range for $m$, especially for small values for $m$ in 3D, the main line of reasoning exemplified below should give a sufficiently broad picture to instill an ample sense of accuracy.

\begin{lemma}\label{lem:approx-sol-prop}
Let $m>1$ and assume that $(u_0,v_0)$ satisfy \eqref{eq:IR}. Then, for each $\epsi\in(0,1)$ there is a pair of functions
\begin{align*}
(\uep,\vep)\in\big(\CSpnl{0}{\bomega\times[0,\infty)}\cap\CSpnl{2,1}{\bomega\times(0,\infty)}\big)^2
\end{align*}
solving \eqref{eq:approx-prop} classically in $\Omega\times(0,\infty)$. Moreover, there exists $M>0$ such that for all $\epsi\in(0,1)$ 
\begin{align}\label{eq:approx-sol-win-est-M}
\|\uep(\cdot,t)\|_{\Lo[\infty]}+\|\vep(\cdot,t)\|_{\W[1,\infty]}\leq M\quad\text{for all }\epsi\in(0,1)\text{ and }t>0.
\end{align}
Furthermore, there are functions
\begin{align*}
u&\in\LSp{\infty}{\Omega\times(0,\infty)},\qquad v\in\LSpb{\infty}{(0,\infty);\W[1,\infty]}
\end{align*}
with $\nabla u^m\in\LSplocb{2}{\bomega\times[0,\infty)}$ and $(\epsi_j)_{j\in\N}\subset(0,1)$ with $\epsi_j\searrow 0$ as $j\to\infty$ such that  $(\uep,\vep)\to (u,v)$ a.e. in $\Omega\times(0,\infty)$ as $j\to \infty$. The functions $u$ and $v$ satisfy the equations
\begin{align}\label{eq:w-sol-prot-eq-u}
-\intinfomega u\varphi_t-\intomega u_0\varphi(\cdot,0)=-\intinfomega\nabla u^m \cdot\nabla\varphi+\intinfomega u\nabla v\cdot\nabla\varphi,
\end{align}
and
\begin{align}\label{eq:w-sol-win-eq-v}
-\intinfomega v\varphi_t-\intomega v_0\varphi(\cdot,0)=-\intinfomega\nabla v\cdot\nabla\varphi-\intinfomega uv\varphi
\end{align}
for every $\varphi\in C^{\infty}_0\big(\bomega\times[0,\infty)\big),$ form a global weak solution of \eqref{eq:CT-prototype-choice} and are globally bounded in the sense that
\begin{align}\label{eq:w-sol-win-est-M}
\|u(\cdot,t)\|_{\Lo[\infty]}+\|v(\cdot,t)\|_{\Lo[\infty]}\leq M\quad\text{for all }t>0.
\end{align}
\end{lemma}

\begin{bew}
Since these results are well-established, we omit most details of the proof and encourage the reader to confer with the original works \cite{caoGlobalintimeBoundedWeak2014,Win-ct_fluid_3d-CPDE15,winklerChemotaxisStokesInteractionVery2022} for additional details regarding global solvability of the approximate systems and the construction of the limit solution. Nevertheless, let us briefly mention some details on the approach for \eqref{eq:approx-sol-win-est-M}, as this is a crucial ingredient for the refined a priori estimates we need to prepare later on.\smallskip

In the 2D setting or the 3D setting with $m>\frac{7}{6}$ the basic idea is to test the first equation against $(\uep+\epsi)^{p-1}$ for $p>1$ and the second equation against $|\nabla\vep|^{2(q-1)}$ for $q>1$. Since the consumption of signal in the second equation implies that there is $K>0$ such that $\|\vep(\cdot,t)\|_{\Lo[\infty]}\leq K$ for all $t>0$, one can derive that with some $C_1=C_1(p)>0$, $C_2=C_2(p)>0$, $C_3=C_3(q)>0$, $C_4>0$ and $C_5=C_5(q)>0$ we have
\begin{align}\label{eq:3.1-eq1}
&\frac{\intd}{\intd t} \intomega(\uep+\epsi)^p+C_1\intomega\big|\nabla(\uep+\epsi)^\frac{m+p-1}{2}\big|^2\leq C_2\intomega(\uep+\epsi)^{p-m+1}|\nabla\vep|^2\quad\text{and}\\\label{eq:3.1-eq2}
&\frac{\intd}{\intd t} \intomega |\nabla\vep|^{2q}+C_3\intomega \big|\nabla|\nabla\vep|^q\big|^2+C_4\intomega|\nabla\vep|^{2q-2}\big|D^2\vep\big|^2\leq C_5\intomega \uep^2|\nabla\vep|^{2q-2}
\end{align}
for all $t>0$ and $\epsi\in(0,1)$. Note that in this process we always use integration by parts to move the derivative away from $\nabla\cdot\big(\frac{\uep}{1+\epsi\uep}\nabla\vep\big)$ and then estimate $\frac{1}{1+\epsi\uep}\leq 1$ to get rid of the approximation term in the chemotactic sensitivity.
Thereafter, using the Hölder and Gagliardo--Nirenberg inequalities to estimate the terms on the right-hand side, one derives a differential inequality of the form 
\begin{align}\label{eq:3.1-diff-ineq}
y_\epsi'(t)+\frac{1}{\mu} y_\epsi(t)\leq \mu 
\end{align}
for $y_\epsi(t):= \intomega\,(\uep(\cdot,t)+\epsi)^p+\intomega |\nabla\vep(\cdot,t)|^{2q}$ and constant $\mu=\mu(p,q)>0$. Let us illustrate some steps of this estimation for the $3D$ case with $m>\frac{7}{6}$ (cf. \cite{Win-ct_fluid_3d-CPDE15}):\smallskip

Starting with the term on the right-hand side of \eqref{eq:3.1-eq1}, we utilize the Hölder inequality to estimate
\begin{align}\label{eq:3.1-eq3}
C_2\intomega(\uep+\epsi)^{p-m+1}|\nabla\vep|^2&\leq C_2\Big(\intomega(\uep+\epsi)^\frac{(p-m+1)(q+1)}{q}\Big)^\frac{q}{q+1}\Big(\intomega|\nabla\vep|^{2q+2}\Big)^\frac{1}{q+1}\nonumber\\
&=C_2\big\|(\uep+\epsi)^\frac{m+p-1}{2}\big\|_{\Lo[\frac{2(p-m+1)(q+1)}{(m+p-1)q}]}^\frac{2(p-m+1)}{m+p-1}\big\|\nabla\vep\big\|_{\Lo[2q+2]}^2
\end{align}
for all $t>0$ and $\epsi\in(0,1)$. Here, employing the \GNI\ to the first term entails the existence of $C_6=C_6(p,q)>0$ satisfying
\begin{align*}
&\Big\|(\uep+\epsi)^\frac{m+p-1}{2}\Big\|_{\Lo[\frac{2(p-m+1)(q+1)}{(m+p-1)q}]}^\frac{2(p-m+1)}{m+p-1}\\\leq\ &C_6\Big\|\nabla(\uep+\epsi)^\frac{m+p-1}{2}\Big\|_{\Lo[2]}^{\frac{2(p-m+1)}{m+p-1}\cdot a}\Big\|(\uep+\epsi)^\frac{m+p-1}{2}\Big\|_{\Lo[\frac{2}{m+p-1}]}^{\frac{2(p-m+1)}{m+p-1}\cdot(1-a)}+C_6\Big\|(\uep+\epsi)^\frac{m+p-1}{2}\Big\|_{\Lo[\frac{2}{m+p-1}]}^\frac{2(p-m+1)}{m+p-1}
\end{align*}
for all $t>0$ and $\epsi\in(0,1)$, with $a=\frac{3(m+p-1)}{p-m+1}\cdot\frac{(q+1)(p-m+1)-q}{3m+3p-2}$. Accordingly, by conservation of mass, we can find $C_7=C_7(p,q)>0$ such that
\begin{align}\label{eq:3.1-eq4}
\Big\|(\uep+\epsi)^\frac{m+p-1}{2}\Big\|_{\Lo[\frac{2(p-m+1)(q+1)}{(m+p-1)q}]}^\frac{2(p-m+1)}{m+p-1}&\leq C_7\Big(\intomega\big|\nabla(\uep+\epsi)^\frac{m+p-1}{2}\big|^2+1\Big)^{\frac{2(p-m+1)}{m+p-1}\cdot a}
\end{align}
for all $t>0$ and $\epsi\in(0,1)$. Similarly, relying on a more specialized variant of the \GNI\ (see e.g. \cite[Lemma 3.8]{Win-ct_fluid_3d-CPDE15}) for the second term on the right-hand side of \eqref{eq:3.1-eq3}, one finds $C_8=C_8(p,q)>0$ and $C_9=C_9(p,q)>0$ satisfying
\begin{align}\label{eq:3.1-eq5}
\big\|\nabla\vep\big\|_{\Lo[2q+2]}^2&\leq \Big\| |\nabla\vep|^{q-1} D^2\vep\Big\|_{\Lo[2]}^\frac{2}{q+1}\big\|\vep\big\|_{\Lo[\infty]}^\frac{2}{q+1}+C_8\big\|\vep\big\|_{\Lo[\infty]}^2\nonumber\\
&\leq C_9\Big(\intomega|\nabla\vep|^{2q-2} |D^2\vep|^2+1\Big)^\frac{1}{q+1}\quad\text{for all }t>0\text{ and }\epsi\in(0,1),
\end{align}
in light of the boundedness of $\vep$. Combining \eqref{eq:3.1-eq3}--\eqref{eq:3.1-eq5} and employing Young's inequality we first find that for any $\eta>0$ there is $C_{10}=C_{10}(p,q,\eta)>0$ such that
\begin{align*}
&C_2\intomega(\uep+\epsi)^{p-m+1}|\nabla\vep|^2\\
\leq\ &\eta\Big(\intomega|\nabla\vep|^{2q-2}|D^2\vep|^2+1\Big)+C_{10}\Big(\intomega\big|\nabla(\uep+\epsi)^\frac{m+p-1}{2}\big|^2+1\Big)^{\frac{2(p-m+1)}{m+p-1}\cdot\frac{q+1}{q}\cdot a}
\end{align*}
for all $t>0$ and $\epsi\in(0,1)$. Restricting to $q>1$ and $p>\max\{1,m-1\}$ satisfying $p<(2m-\tfrac43)q+m-1$ we have $$\frac{2(p-m+1)}{m+p-1}\cdot\frac{q+1}{q}\cdot a<1$$ and can draw on Young's inequality an additional time to find that with some $C_{11}=C_{11}(p,q,\eta)>0$ 
\begin{align}\label{eq:3.1-eq6}
C_2\intomega(\uep+\epsi)^{p-m+1}|\nabla\vep|^2\leq \eta\Big(\intomega|\nabla\vep|^{2q-2}|D^2\vep|^2+1\Big)+\eta\Big(\intomega\big|\nabla(\uep+\epsi)^\frac{m+p-1}{2}\big|^2+1\Big)+C_{11}
\end{align}
holds for all $t>0$ and $\epsi\in(0,1)$. Arguing similarly, restricting to $p,q>1$ satisfying $p>\frac{3p-3m+4}{3}$ one can also show that there is $C_{12}=C_{12}(p,q,\eta)>0$ such that
\begin{align}\label{eq:3.1-eq7}
C_5\intomega\uep^2|\nabla\vep|^{2q-2}\leq \eta\Big(\intomega|\nabla\vep|^{2q-2}|D^2\vep|2+1\Big)+\eta\Big(\intomega\big|\nabla(\uep+\epsi)^\frac{m+p-1}{2}\big|2+1\Big)+C_{12}
\end{align}
for all $t>0$ and $\epsi\in(0,1)$. Here, direct calculation show that the admissible range for $p$ to make use of both \eqref{eq:3.1-eq6} and \eqref{eq:3.1-eq7} is not empty under the assumption $m>\frac{7}{6}$, and picking $\eta>0$ sufficiently small, we conclude from gathering \eqref{eq:3.1-eq1}, \eqref{eq:3.1-eq2}, \eqref{eq:3.1-eq6} and \eqref{eq:3.1-eq7} that there is $C_{13}=C_{13}(p,q,\eta)>0$ such that
\begin{align}\label{eq:3.1-eq8}
y'_\epsi(t)+ \frac{1}{C_{13}}\Big(\intomega\big|\nabla(\uep+\epsi)\big|^\frac{m+p-1}{2}+\intomega\big|\nabla|\nabla\vep|^q\big|^2+\intomega|\nabla\vep|^{2q-2}|D^2\vep|^2\Big)\leq C_{13}
\end{align}
holds for all $t>0$ and $\epsi\in(0,1)$. Hence, estimating terms on the left-hand side one last time from below, we infer the postulated differential inequality for $y_\epsi$ as described in \eqref{eq:3.1-diff-ineq}.

The estimations in the two-dimensional setting follow the same philosophy (\cite{caoGlobalintimeBoundedWeak2014}), with the exponents of the Gagliardo--Nirenberg inequalities providing more leniency for the range of $m$ in this case.\smallskip

Having the autonomous differential inequality in \eqref{eq:3.1-diff-ineq} at hand, we can use a comparison argument to extract new bounds on $\uep$ and $\nabla\vep$, which can then be exploited by a Trudinger--Moser type iteration to establish the bound for $\uep$ in $\Lo[\infty]$. Whereafter, drawing on semigroup estimates for the Neumann heat semigroup, we can also conclude the $\W[1,\infty]$-bound for $\vep$. Additionally, integrating the differential inequality in \eqref{eq:3.1-eq8} also provides spatio-temporal bounds for the gradient terms, which in combination with easily established spatio-temporal bounds on the time derivatives allow an application of an Aubin--Lions type lemma to obtain limit functions of the desired regularity, which also satisfy \eqref{eq:w-sol-prot-eq-u}--\eqref{eq:w-sol-win-est-M}.\smallskip

In 3D with small values of $m$, i.e. $m\in(1,2]$, the differential inequality has to be prepared in a more intricate way (\cite{winklerChemotaxisStokesInteractionVery2022}). In fact, one introduces $D_\epsi(u)=\frac{1}{m}(\uep+\epsi)^{m-1}$, $D_{1,\epsi}(s)=\int_0^{s} D_{\epsi}(\sigma)\intd \sigma$ and $D_{2,\epsi}(u)=\int_0^{u} D_{1,\epsi}(s)\intd s$, as well as 
$$\psi_{0,\epsi}^{(s_0)}(u)=\begin{cases}
\frac{1}{D_\epsi(u)},\quad &u\in(0,s_0),\\
\frac{2s_0-u}{s_0 D_\epsi(s_0)},\quad &u\in[s_0,2s_0],\\
0,\quad &u>2s_0
\end{cases} $$
for $s_0>0$, $\psi_{1,\epsi}^{(s_0)}(u)=\int_{2s_0}^u\psi_{0,\epsi}^{(s_0)}(s)\intd s$ and $\psi_{2,\epsi}^{(s_0)}(u)=\int_{2s_0}^u\psi_{1,\epsi}^{(s_0)}(s)\intd s$ and (cf. \cite[Lemmas 3.1-3.5]{winklerChemotaxisStokesInteractionVery2022}) calculates 
\begin{align*}
&\frac{\intd}{\intd t}\intomega D_{2,\epsi}(\uep)+\frac{1}{2}\intomega D_\epsi^2(\uep)|\nabla\uep|^2\leq c_1\intomega\frac{\uep^2}{\vep}|\nabla\vep|^2,\\
&\frac{\intd}{\intd t}\intomega\frac{\uep|\nabla\vep|^2}{\vep}+\frac{1}{2}\intomega\frac{\uep|\nabla\vep|^2}{\vep}+2\intomega\frac{\uep}{\vep}\big|D^2\vep\big|^2+2\intomega\frac{\uep}{\vep^3}|\nabla\vep|^4\\
&\qquad\qquad\leq \eta \intomega D_\epsi^2(\uep)|\nabla\uep|^2+c_2(\eta,K)\bigg(\intomega|\nabla\uep|^2+\intomega\frac{\big|D^2\vep\big|^2|\nabla\vep|^2}{\vep^3}+\intomega\frac{|\nabla\vep|^6}{\vep^5}\bigg),\\
&\frac{\intd}{\intd t}\intomega \frac{|\nabla\vep|^4}{\vep^3}+\gamma\intomega\frac{|\nabla\vep|^6}{\vep^5}+\gamma\intomega\frac{|\nabla\vep|^2}{\vep^3}\big|D^2\vep\big|^2\leq c_3(K)\intomega|\nabla\uep|^2,\intertext{and}
&\frac{\intd}{\intd t}\intomega \psi_{2,\epsi}(\uep)+\frac{1}{2}\int_{\{\uep<s_0\}}\big|\nabla\uep\big|^2\leq c_4(K,s_0)\intomega\frac{|\nabla\vep|^2}{\vep}
\end{align*}
with positive constants $c_i>0$ and a certain constant $\gamma>0$. Choosing $s_0$ and $\eta$ properly, one obtains a differential inequality of structural similarity to before, but with the additional benefit of featuring $\intomega |\nabla\uep|^2$ on the left-hand side. From this one can refine to additional bounds, which once more fuel an Aubin--Lions type argument to conclude the existence of the limit functions satisfying the weak solution concept.
\end{bew}

For the remainder of this work we let $m>1$ and fix initial data $u_0,v_0$ complying with \eqref{eq:IR}. To establish that the corresponding limit solution, in addition to the properties described in Lemma~\ref{lem:approx-sol-prop}, also satisfies the quintessential condition \eqref{eq:cond-conti}, we require supplementary precompactness properties, which go beyond the scope of the cited works. As first step in our pursuit in this matter, we prove a spatio-temporal bound on $(\uep+\epsi)^{m+p-3}|\nabla\uep|^2$.

\begin{lemma}\label{lem:st-bound-grad-uep}
Let $p>m-1$. Denote by $(\uep,\vep)$ the classical solution to \eqref{eq:approx-prop} with characteristics as described in Lemma~\ref{lem:approx-sol-prop}. Then, there is $C=C(p)>0$ such that
\begin{align*}
\intoTomega (\uep+\epsi)^{m+p-3}|\nabla\uep|^2+\intoTomega|\nabla\vep|^2+\intoTomega\uep\vep^2\leq C
\end{align*}
holds for all $T>0$ and $\epsi\in(0,1)$.
\end{lemma}

\begin{bew}
Due to the nonnegativity of $\vep$ we easily obtain from testing the second equation of \eqref{eq:approx-prop} against $\vep$ and integrating by parts that
\begin{align}\label{eq:st-bound-grad-uep-eq1}
\intoTomega|\nabla\vep|^2+\intoTomega\uep\vep^2\leq \frac{1}{2}\intomega v_0^2
\end{align}
holds for all $T>0$ and $\epsi\in(0,1)$. Similarly, testing the first equation of \eqref{eq:approx-prop} against $(\uep+\epsi)^{p-1}$ and integrating by parts, we find that
\begin{align*}
\frac{1}{p}\frac{\intd}{\intd t}\intomega (\uep+\epsi)^{p}&=-m(p-1)\intomega (\uep+\epsi)^{m+p-3}|\nabla\uep|^2+(p-1)\intomega\frac{\uep}{1+\epsi\uep}(\uep+\epsi)^{p-2}\big(\nabla\vep\cdot\nabla\uep\big)
\end{align*}
holds for all $t>0$ and $\epsi\in(0,1)$. Employing Young's inequality in the last integral, we conclude that
\begin{align*}
\frac{1}{p}\frac{\intd}{\intd t}\intomega (\uep+\epsi)^{p}+\frac{m(p-1)}{2}\intomega(\uep+\epsi)^{m+p-3}|\nabla\uep|^2&\leq \frac{p-1}{2m}\intomega\uep^2(\uep+\epsi)^{p-m-1}|\nabla\vep|^2
\end{align*}
is valid for all $t>0$ and $\epsi\in(0,1)$, which by integrating from $0$ to $T$ entails that
\begin{align*}
\intoTomega(\uep+\epsi)^{m+p-3}|\nabla\uep|^2\leq \frac{1}{p}\intomega(u_0+1)^{p}+\frac{p-1}{2m}\intoTomega (\uep+1)^{p-m+1}|\nabla\vep|^2
\end{align*}
for all $T>0$ and $\epsi\in(0,1)$. In light of $p-m+1>0$ and $\|\uep(\cdot,t)\|_{\Lo[\infty]}\leq M$ for all $t>0$, we get
\begin{align}\label{eq:st-bound-grad-uep-eq2}
\intoTomega(\uep+\epsi)^{m+p-3}|\nabla\uep|^2\leq \frac{1}{p}\intomega(u_0+1)^{p}+\frac{p-1}{2m}(M+1)^{p-m+1}\intoTomega |\nabla\vep|^2
\end{align}
for all $T>0$ and $\epsi\in(0,1)$. The conclusion of the proof is now an evident combination of \eqref{eq:st-bound-grad-uep-eq1}, \eqref{eq:st-bound-grad-uep-eq2} and \eqref{eq:ID-bounds}.
\end{bew}

In the next step we derive some additional quantifiable information on a quantity involving the time-derivatives of $\uep$ and $\vep$.

\begin{lemma}\label{lem:bound-time-deriv-uep}
 Denote by $(\uep,\vep)$ the classical solution to \eqref{eq:approx-prop} obtained in Lemma~\ref{lem:approx-sol-prop}. There is $C>0$ such that
\begin{align*}
\intomega\big|\nabla\uep^m(\cdot,T)\big|^2+\intoTomega (\uep+\epsi)^{m-1}\big(u_{\epsi t}\big)^2+\intoTomega\big(v_{\epsi t}\big)^2\leq C
\end{align*}
holds for all $T>0$ and $\epsi\in(0,1)$.
\end{lemma}

\begin{bew}
Direct calculations drawing on the second equation of \eqref{eq:approx-prop} and utilizing integration by parts and Young's inequality show that
\begin{align*}
\frac{\intd}{\intd t}\intomega|\nabla\vep|^2+\intomega|\Delta\vep|^2\leq 2\intomega\uep^2\vep^2
\end{align*}
is valid for all $t>0$ and $\epsi\in(0,1)$. In view of \eqref{eq:ID-bounds}, \eqref{eq:approx-sol-win-est-M} and Lemma~\ref{lem:st-bound-grad-uep} we conclude, that there is $C_1>0$ such that
\begin{align}\label{eq:st-bound-time-deriv-uep-eq1}
\intoTomega|\Delta\vep|^2\leq \intomega |\nabla v_0|^2+2M\intoTomega\uep\vep^2\leq C_1
\end{align}
holds for all $T>0$ and all $\epsi\in(0,1)$. Next, we test the second equation of \eqref{eq:approx-prop} against $v_{\epsi t}$ and employ Young's inequality to find that
\begin{align*}
\intomega \big(v_{\epsi t}\big)^2=\intomega v_{\epsi t}\big(\Delta\vep-\vep\uep\big)\leq \frac{1}{2}\intomega \big(v_{\epsi t}\big)^2 +\intomega|\Delta\vep|^2+\intomega \uep^2\vep^2
\end{align*}
for all $t>0$ and $\epsi\in(0,1)$. Accordingly, making use of \eqref{eq:st-bound-time-deriv-uep-eq1} and Lemma~\ref{lem:st-bound-grad-uep}, we find $C_2>0$ satisfying
\begin{align}\label{eq:time-u-est-v}
\intoTomega \big(v_{\epsi t}\big)^2 \leq 2\intoTomega|\Delta\vep|^2+2M\intoTomega \uep\vep^2\leq C_2\quad\text{for all }T>0\text{ and }\epsi\in(0,1).
\end{align}
Similarly, testing the first equation of \eqref{eq:approx-prop} against $\big(\uep^m\big)_t$ yields
\begin{align}\label{eq:time-u-est-0}
\intomega u_{\epsi t}\big((\uep+\epsi)^m\big)_{t}\nonumber
=\ &\intomega \nabla\cdot\big(\nabla(\uep+\epsi)^m\big)\big((\uep+\epsi)^m\big)_t\nonumber\\
&-\intomega\frac{\nabla\uep\cdot\nabla\vep}{1+\epsi\uep}\big((\uep+\epsi)^m\big)_t+\intomega\frac{\epsi\uep\big(\nabla\uep\cdot\nabla\vep\big)}{(1+\epsi\uep)^2}\big((\uep+\epsi)^m\big)_t\nonumber\\
&-\intomega \frac{\uep}{1+\epsi\uep}\Delta\vep\big((\uep+\epsi)^m\big)_t=:I_1+I_2+I_3+I_4
\end{align}
for all $t>0$ and $\epsi\in(0,1)$. To treat $I_1$, we integrate by parts to obtain
\begin{align}\label{eq:time-u-est-1}
I_1=-\intomega\nabla(\uep+\epsi)^m\cdot\big(\nabla(\uep+\epsi)^m\big)_t=-\frac{1}{2}\frac{\intd}{\intd t}\intomega \big|\nabla(\uep+\epsi)^m\big|^2
\end{align}
for all $t>0$ and $\epsi\in(0,1)$. In $I_2$ and $I_3$ we write $((\uep+\epsi)^m)_t=m(\uep+\epsi)^{m-1} u_{\epsi t}$ and employ Young's inequality to estimate
\begin{align*}
I_2+I_3\leq \frac{m}{4}\intomega (\uep+\epsi)^{m-1}(u_{\epsi t})^2&+2m\intomega(\uep+\epsi)^{m-1}|\nabla\uep|^2|\nabla\vep|^2\\
&+4m\intomega(\uep+\epsi)^{m-1}\frac{\epsi^2\uep^2}{(1+\epsi\uep)^2}|\nabla\uep|^2|\nabla\vep|^2
\end{align*}
for all $t>0$ and $\epsi\in(0,1)$. Since \eqref{eq:approx-sol-win-est-M} implies $\|\vep\|_{\W[1,\infty]}\leq M$ for all $t>0$ and $\epsi\in(0,1)$, we immediately obtain $\|\nabla\vep\|_{\Lo[\infty]}\leq M$ for all $t>0$ and $\epsi\in(0,1)$ and hence find that
\begin{align}\label{eq:time-u-est-2}
I_2+I_3\leq \frac{m}{4}\intomega(\uep+\epsi)^{m-1}(u_{\epsi t})^2+ 4m M^2\intomega(\uep+\epsi)^{m-1}|\nabla\uep|^2\quad\text{for all }t>0\text{ and }\epsi\in(0,1).
\end{align}
Similarly, an application of Young's inequality to $I_4$ yields
\begin{align*}
I_4\leq \frac{m}{4}\intomega(\uep+\epsi)^{m-1}\big(u_{\epsi t})^2+ m\intomega\frac{(\uep+\epsi)^{m-1}\uep^2}{(1+\epsi\uep)^2}|\Delta\vep|^2\quad\text{for all }t>0\text{ and }\epsi\in(0,1),
\end{align*}
and drawing on \eqref{eq:approx-sol-win-est-M} again to estimate $\|\uep\|_{\Lo[\infty]}\leq M$ for all $t>0$ and $\epsi\in(0,1)$, entails
\begin{align}\label{eq:time-u-est-3}
I_4\leq \frac{m}{4}\intomega(\uep+\epsi)^{m-1}\big(u_{\epsi t})^2+m (M+1)^{m+1}\intomega|\Delta\vep|^2\quad\text{for all }t>0\text{ and }\epsi\in(0,1).
\end{align} 
Combining \eqref{eq:time-u-est-0}--\eqref{eq:time-u-est-3} and integrating from $0$ to $T$, we obtain
\begin{align*}
&\frac{m}{2}\intoTomega(\uep+\epsi)^{m-1}\big(u_{\epsi t})^2+\frac{1}{2}\intomega\big|\nabla\big(\uep(\cdot,T)+\epsi\big)^m\big|^2\nonumber\\
\leq\ &4mM^2\intoTomega(\uep+\epsi)^{m-1}|\nabla\uep|^2+m(M+1)^m\intoTomega|\Delta\vep|^2+\frac{1}{2}\intomega\big|\nabla (u_0+\epsi)^m\big|^2
\end{align*}
for all $T>0$ and $\epsi\in(0,1)$. In view of Lemma~\ref{lem:st-bound-grad-uep} and the bounds featured in \eqref{eq:time-u-est-v} and \eqref{eq:ID-bounds}, we readily conclude that there is $C_3>0$ such that
\begin{align}\label{eq:time-u-est-4}
\intomega\big|\nabla\uep^m(\cdot,T)\big|^2+\intoTomega(\uep+\epsi)^{m-1}\big(u_{\epsi t})^2\leq C_3\quad\text{for all }T>0\text{ and }\epsi\in(0,1),
\end{align}
where we have used that $\intomega|\nabla\uep^m|^2\leq \intomega|\nabla(\uep+\epsi)^m|^2$. The inequalities \eqref{eq:time-u-est-4} and \eqref{eq:time-u-est-v} together prove the assertion.
\end{bew}

The lemmas above at hand, we can now draw on an Aubin–-Lions–-Simon compactness lemma to conclude the following.

\begin{lemma}\label{lem:simon}
Up to a re-definition on a set of measure zero the functions $u,v$ provided by Lemma~\ref{lem:approx-sol-prop} satisfy
\begin{align*}
u^m\in\CSp{0}{[0,\infty);\Lo[2]},\quad v\in\CSp{0}{[0,\infty);\Lo[2]}.
\end{align*}
\end{lemma}

\begin{bew}
Let $T>0$ be arbitrary. Then, according to Lemma~\ref{lem:bound-time-deriv-uep}, we find that $$\big(\uep^m\big)_{\epsi\in(0,1)}\ \text{ is bounded in }\ \LSp{\infty}{(0,T);\W[1,2]}$$ and $$\big(u_{\epsi t}^m\big)_{\epsi\in(0,1)}\ \text{ is bounded in }\ \LSp{2}{(0,T);\Lo[2]}.$$ Drawing on e.g. \cite[Theorem 1]{simonCompactSetsSpace1987}, we conclude that $$u^m\in\CSp{0}{[0,T];\Lo[2]}.$$ Combining the boundedness $(v_{\epsi t})_{\epsi\in(0,1)}$ in $\LSp{2}{(0,T);\Lo[2]}$ established in Lemma~\ref{lem:bound-time-deriv-uep} with \eqref{eq:approx-prop}, we also arrive at the corresponding statement for $v$.
\end{bew}

\begin{proof}[\textbf{Proof of Theorem~\ref{theo2}:}]
With $(u,v)$ denoting the global weak solution of \eqref{eq:CT-prototype-choice} on $\Omega\times(0,\infty)$ provided by Lemma~\ref{lem:approx-sol-prop}, we set $T_0:=\infty$ and notice that for $T\in(0,T_0)$ we have
\begin{align*}
\big\|u(\cdot,t)-u(\cdot,s)\big\|_{\Lo[2]}^2\leq |\Omega|^{1-\frac{1}{m}}\big\|u^m(\cdot,t)-u^m(\cdot,s)\big\|_{\Lo[2]}^\frac{2}{m}\quad\text{for all }s,t\in[0,T],
\end{align*}
due to $m>1$ and $u\geq0$. Consequently, we conclude directly from Lemma~\ref{lem:simon} that both
\begin{align*}
u&\in C_{loc}^0\big([0,\infty);\LSp{2}{\Omega}\!\big)
,\quad\text{and}\quad
v\in C_{loc}^0\big([0,\infty);\LSp{2}{\Omega}\!\big)
\end{align*}
are satisfied. Combining these properties with the bounds established in \eqref{eq:w-sol-win-est-M}, we find that Theorem~\ref{theo1} applies to $(u,v)$, which in turn ensures that the claims of Theorem~\ref{theo2} are true.
\end{proof}

\section*{Acknowledgements}
The author acknowledges support of the {\em Deutsche~Forschungsgemeinschaft} (Project No.~462888149).

\footnotesize{
\setlength{\bibsep}{3pt plus 0.5ex}

}

\end{document}